\documentclass[a4paper]{article}
\usepackage{amsthm,amsmath,amssymb}
\usepackage{enumerate,enumitem}
\usepackage{graphicx}
\usepackage{xfrac}

\usepackage{tabularx,booktabs,multirow}
\newcolumntype{C}{>{\centering\arraybackslash}X}
\newcolumntype{D}{>{\centering\arraybackslash}X}

\DeclareGraphicsExtensions{.pdf,.png,.eps}
\graphicspath{{figs/}}

\newtheorem{theorem}{Theorem}
\newtheorem{lemma}[theorem]{Lemma}

\newtheorem{corollary}[theorem]{Corollary}

\newtheorem*{claim*}{Claim}

\theoremstyle{remark}

\newcommand{\G}{\ensuremath{\mathcal{G}}}

\newcommand{\cA}{\ensuremath{\mathcal{A}}}
\newcommand{\B}{\ensuremath{\mathcal{B}}}

\newcommand{\cF}{\ensuremath{\mathcal{F}}}

\newcommand{\cS}{\ensuremath{\mathcal{S}}}

\newcommand{\cX}{\ensuremath{\mathcal{X}}}

\usepackage[usenames,dvipsnames]{xcolor}

\begin{document}

\title{A note on Ramsey numbers for Berge-$G$ hypergraphs}

\author{Maria Axenovich\thanks{Karlsruhe Institute of Technology, Karlsruhe, Germany}\and Andr\'as Gy\'arf\'as\thanks{Renyi  Institute of Mathematics, Budapest, Hungary}\thanks{Research supported in part by
NKFIH Grant No. K116769.}}

\maketitle

\begin{abstract}
For a graph $G=(V,E)$,  a hypergraph $H$ is called {\it Berge-$G$} if there is a bijection $\phi: E(G) \rightarrow E(H)$ such that for each $e\in E(G)$, $e \subseteq \phi(e)$.  The set of all Berge-$G$ hypergraphs is denoted $\B(G)$.

For integers $k\geq 2$, $r\geq 2$, and a graph $G$, let the Ramsey number $R_r(\B(G), k)$ be the smallest integer $n$ such that no matter how the edges of a complete $r$-uniform $n$-vertex hypergraph are colored with $k$ colors, there is a copy of a monochromatic Berge-$G$ subhypergraph. Furthermore, let $R(\B(G),k)$ be the smallest integer $n$ such that no matter how all subsets an $n$-element set are colored with $k$ colors, there is a monochromatic copy of a Berge-$G$ hypergraph.

We give an upper bound for $R_r(\B(G),k)$ in terms of graph Ramsey numbers. In particular, we prove that when $G$ becomes acyclic after removing some vertex, $R_r(\B(G),k)\le 4k|V(G)|+r-2$, in contrast with classical multicolor Ramsey numbers.

When $G$ is a triangle (or a $K_4$), we find sharper bounds and some exact results and  determine some ``small'' Ramsey numbers:

\begin{itemize}

\item $k/2 - o(k)  \leq R_3(\B(K_3)), k) \leq  3k/4+ o(k)$,

\item For any odd integer $t\ne 3$,  $R(\B(K_3),2^t-1)=t+2$,

\item  $2^{ck} \leq R_3(\B(K_4),k)\leq e(1+o(1))(k-1)k!$,

\item $R_3(\B(K_3),2)=R_3(\B(K_3),3)=5,~ R_3(\B(K_3),4)=6,~ R_3(\B(K_3),5)=7,~R_3(\B(K_3),6)=8, ~R_3(\B(K_3,8)=9, ~R_3(\B(K_4),2)=6$.
\end{itemize}
\end{abstract}

\section{Introduction}

For a graph $G$, a family  $\B(G)$  consists of hypergraphs  $H$ each with $|E(G)|$ distinct hyperedges so that
for each $xy\in E(G)$,  there is a hyperedge $e_{xy}$ of $H$ such that
 $\phi(x),\phi(y) \in  e_{xy}$ and $e_{xy}\neq e_{x'y'}$ if $xy\neq x'y'$, for an injective map $\phi: V(G)\rightarrow V(H)$.
 Here, we shall always denote the vertex set of $F$ as $V(F)$ and the edge set of $F$ as $E(F)$, for a graph or a hypergraph $F$. A copy of a graph $F$ in a graph $G$ is a subgraph of $G$ isomorphic to $F$. When clear from context, we shall drop the word ``copy" and just say that there is $F$ in $G$.

The members of $\B(G)$ are called {\it Berge-$G$ hypergraphs}.
We call  a copy $G'$ of $G$, where $G'= (\phi(V(G)), \{\phi(x)\phi(y):  xy\in E(G)\})$,  the {\it underlying  graph} of the Berge-$G$ hypergraph.

The Ramsey number $R_r(\cF, k)$ is the smallest integer $n$ such that no matter how the edges of a complete $r$-uniform $n$-vertex hypergraph, that we denote $K^r_n$,  are colored with $k$ colors, there is a monochromatic  subhypergraph from $\cF$. We always assume that $k\geq 2$, $r\geq 2$, and simply write $K_n$ for $K^2_n$. The classical $k$-color Ramsey number for a graph $G$ (uniformity $2$) is denoted $R(G, k)$, i.e., $R(G, k) = R_2(\{G\}, k)$.  When we do not restrict our attention to uniform hypergraphs, we define the Ramsey number $R(\B(G), k)$ to be the smallest integer $n$ such that no matter how all subsets an $n$-element set are colored with $k$ colors, there is a monochromatic copy of a Berge-$G$ hypergraph.
 It is convenient to define dual functions:  $f(n, \B(G))$, the smallest number of colors in a coloring of $2^{[n]}$ such that there is no monochromatic Berge-$G$ hypergraph and  $f_r(n, \B(G))$, the smallest number of colors in a coloring of $\binom{[n]}{r}$ such that there is no monochromatic Berge-$G$ hypergraph.


For $G=C_t$, the cycle on $t$ vertices,  and $r=3$, this problem has been already investigated.  It was proved by Gy\'arf\'as,  Lehel, S\'ark\"ozy, and  Schelp \cite{GYLSS} that $R_3(\B(C_t),2)=t$ for $t\ge 5$. The fact that  $R_3(\B(C_t),3)\sim {5t\over 4}$ was the main result of the paper by Gy\'arf\'as and  S\'ark\"ozy, \cite{GYS}.\\

The Ramsey problem is closely related to Tur\'an problems. For a family $\cF$ of hypergraphs,  the Tur\'an number ${\rm ex}_r(n, \cF)$ is the largest number of edges in an $r$-uniform hypergraph on $n$ vertices that does not contain any member from the family $\cF$ as a subhypergraph.  Indeed, let $N=R_r(\cF, k) -1$.  Since in a coloring of $K_N^r$ in $k$ colors with no monochromatic member of $\cF$ each color class has at most  ${\rm ex}_r(N, \cF)$ edges, we have
\begin{equation}
\binom{N}{r}  \leq k \cdot {\rm ex}_r(N, \cF).
\end{equation}\label{upper-bound-extremal}

Note that this inequality gives easy upper bounds on $N$  if ${\rm ex}_r(N, \cF) = o(N^r)$. Gy\H{o}ri \cite{G} proved that  ${\rm ex}_3(n, \B(K_3))\leq n^2/8$ and the bound is tight. The non-uniform and multi-hypergraph case was addressed in that paper too.  Other results on extremal numbers for Berge hypergraphs were provided by Gerbner and Palmer \cite{GP},
 Gerbner, Methuku, and Vizer, \cite{GMV}, Gerbner, Methuku, and Palmer, \cite{GMP},
as well as by Palmer,  Tait, Timmons, and Wagner \cite{PTTW}, and by Gr\'osz, Methuku, and Tompkins \cite{GMT}.\\

We provide bounds on  uniform and non-uniform Ramsey numbers for Berge-$K_3$ hypergraphs, including several exact results.  We also give results for Ramsey numbers of Berge hypergraphs for  general graphs.

\begin{theorem}\label{triangle} Let $\cF=\B(K_3)$. Then for $n\geq 3$
\begin{enumerate}
\item{}$k/2 - o(k)  \leq R_3(\cF), k) \leq  3k/4+ o(k)$,
\item{} $R_3(\cF,2)=R_3(\cF,3)=5,~R_3(\cF,4)=6,~R_3(\cF,5)=7,~
  R_3(\cF,6)=8, ~R_3(\cF,8)=9$,
\item{}$2^{n-2}(1-o(1))  \leq f(n, \cF)\leq 2^{n-2}$,
\item{} $f(n, \cF)= 2^{n-2}$, for odd $n\ge 3$, except for $n=5$. 
\end{enumerate}
\end{theorem}

We make a connection between $r$-uniform Ramsey numbers of Berge-$G$ hypergraphs, $r\geq 3$,  and multicolor Ramsey numbers for auxiliary families of graphs.
For a graph $G=(V,E)$ and $v\in V$,  $\G^*(v)$ is defined as the class of  all graphs obtained from $G$  by the following procedure. Let $N(v)=\{q_1,\dots,q_t\}$ denote the set of vertices adjacent to $v$ in $G$. Let $G'=G-v$ be the graph obtained from $G$ by deleting $v$ and the edges of $G$ incident to $v$.  Then, for every $q_i\in N(v)$  add a {\em new edge} $q_ir_i$ (not in $G'$) where $r_i$ can be any vertex of $G'$ or any new vertex (not in $G'$).
These $|N(v)|$ new edges could be pendant, forming a matching, or could share endpoints. Thus $\G^*(v)$ includes $G$ and many other graphs. The graph obtained this way is denoted by $G''(v;q_1r_1,\dots,q_tr_t)$ and is called an {\it extension} of $G-v$.
 For example, if  $G= K_4-e$, the graph obtained from $K_4$ by deleting an edge,  and $v$ is a vertex of degree three in it, $\G^*(v)$ consists of $G$ and four other graphs. When $G=K_4$, $\G^*(v)$ consists of $G$, $K_4-e$ with a pendant edge incident to a vertex of degree $2$ of $K_4-e$, and a triangle with three pendant edges. See
 Figure \ref{G*}.

\begin{figure}[h]\label{G*}
\begin{center}
\end{center}
\caption{A family $\G^*(v)$  for $G= K_4-e$ and for $G=K_4$}
\end{figure}

\begin{theorem}\label{shadowbound}
For any graph $G$, $v\in V(G)$, and any integer $r\geq 3$,  we have $R_r(\B(G),k)\le R(\G^*(v), k)+r-2$.
\end{theorem}

\begin{proof}
Consider a $k$-coloring $c$ of  $K=K_n^r$ on a vertex set $V$,  $|V|=n \geq  R(G^*(v), k)+r-2$.
Fix an arbitrary $(r-2)$-element subset $X\subset V$ and $k$-color the edges of the complete graph $K'$ with vertex set $V\setminus X$ by the rule $c(xy)=c(X\cup \{x,y\})$. From the choice of $n$, we have a monochromatic, say red copy  $G''=G''(v;q_1r_1,\dots,q_tr_t)$ of a member of $\G^*(v)$ in $K'$.  We claim that $K$ contains a red member of $\B(G)$.

To prove the claim, let $x\in X$ and define the graph $F$ as follows. Let $$V(F)=V(G'')\cup \{x\}, E(F)=(E(G'')\cup_{i=1}^t xq_i)\setminus (\cup_{i=1}^t q_ir_i).$$
We show that the edges of $F$ can be covered by distinct red edges of $K$. Indeed,  $xq_i\in (E(F)\setminus E(G''))$ is covered by $\{q_i,r_i\}\cup X$ and any edge $vw\in (E(F)\cap E(G''))$ is covered by $\{v,w\}\cup X$ and these sets are red edges of $K$. Thus the graph obtained from $F$ upon removing those $r_i$-s that are not in $V(G')$ (they are isolated vertices of $F$) is isomorphic to $G$ and its edges are covered by distinct red edges of $K$. This proves the claim and Theorem \ref{shadowbound}. \end{proof}

Since $G\in  \G^*(v)$ for every $v$, Theorem \ref{shadowbound} implies

\begin{corollary}\label{cor1} $R_r(\B(G),k)\le R(G,k)+r-2$.
\end{corollary}

However, a much stronger bound follows from Theorem \ref{shadowbound}.

\begin{corollary}\label{cors} Set $\G=\cup_{v\in V(G)} \G^*(v)$. Then $R_r(\B(G),k)\le R(\G,k)+r-2$.
\end{corollary}

\begin{corollary}\label{cor2}
If a graph $G$ can be made acyclic by the removal of a vertex, then  $R_r(\B(G),k)\le 4k|V(G)|+r-2$ for every $r\geq 3$.  
\end{corollary}
\begin{proof}If $G-v$ is acyclic, $G'=G-v$ has an acyclic extension $G''$ obtained by adding a matching $q_ir_i$ from $N(v)$ to new vertices. Clearly $|V(G'')|<2|V(G)|$ and since $G''$ is acyclic and $R(G'',k)\le 2k|V(G'')|$ (see e.g. \cite{GRS}), Corollary \ref{cor2} follows from Theorem \ref{shadowbound}.
\end{proof}

For the non-uniform case we have the following bounds.

\begin{theorem}\label{non-uniform}
Let $G$ be a graph with at least two edges. If $G\neq 2K_2$, then
\begin{enumerate}
\item{} $$\frac{2^{n-|V(G)|}}{|E(G)|-1}\leq f(n, \B(G))\leq 2^{n-1}.$$
\item{}  In addition, if $G$ has maximum degree at most $2$, then
 $$\frac{2^{n-1}}{|E(G)|-1}(1-o(1))\leq f(n, \B(G)).$$
\item{}  $$f(n, \B(2K_2)) = 2^n - \binom{n}{2}-n-1.$$
\end{enumerate}
\end{theorem}

\noindent
Moreover, we have some results for Ramsey numbers of $\B(K_4)$.
Set $K_4^*=K_4^*(v)$, for a vertex $v$ in $K_4$.

\begin{theorem}\label{quadrangle} We have that for a positive constant $c$
\begin{enumerate}

\item{}$ 2^{ck}    \leq R_3(\B(K_4), k)\leq R(K_4^*, k)+1\leq e(1+o(1))(k-1)k!,$
\item{}$R_3(\B(K_4),2)=6$.
\end{enumerate}
\end{theorem}
%

Note that part $1$ in  Theorem \ref{quadrangle} shows that the upper bound of the multicolor Ramsey number for the family $K^*_4(v)$ differs from the best known upper bound of $R(K_3,k)$ only by a factor  linear in $k$. It is also worth mentioning that part $2$. in Theorem \ref{quadrangle} shows that $R_3(\B(K_4),2)$ is much smaller than its classical counterpart, $R_3(K_4^3,2)=13$, \cite{RAD}.

The rest of the paper is structured as follows. In Section \ref{sec:non-uniform} we treat the non-uniform case proving parts $3$ and $4$ of Theorem \ref{triangle} and Theorem \ref{non-uniform}.  In Section \ref{sec:triangle} we prove the remaining parts $1$ and $2$ of Theorem \ref{triangle}.
Finally, in Section \ref{sec:quadrangle} we prove Theorem \ref{quadrangle}.




\section{The non-uniform case}\label{sec:non-uniform}

%
%
%
%

\smallskip
\noindent
{\bf Proof of Theorem \ref{non-uniform}/1 - upper bound.} If $G$ is not a $2K_2$, make each color class consisting of a set and its complement. This gives a general upper bound.

\smallskip
\noindent
{\bf Proof of Theorem \ref{triangle}/3 - upper bound.} For the upper bound on $f(n, \B(K_3))$, consider the coloring of $2^{[n]}$ such that each color class consists of four sets:   $A$,  $[n-1]-A$, $[n]-A$, and $A\cup \{n\}$ for $A\subseteq [n-1]$. Then the total number of colors is $2^{n-2}$. The four sets of each color class do not contain Berge-$K_3$.

\smallskip
\noindent
{\bf Proof of Theorem \ref{non-uniform}/1 - lower bound.} To prove the bound $f(n, \B(G))\geq \frac{2^{n-|V(G)|}}{|E(G)|-1}$, consider a set $S$ of $|V(G)|$ vertices and the set $\cS$ of  all subsets containing $S$.
Note that any $|E(G)|$ sets from $\cS$  form a Berge-$G$ hypergraph.
Thus there are at most $|E(G)|-1$ members of $\cS$ of each color.
Therefore the total number of colors is at least the number of colors used on $\cS$, that in turn is at least
$$\frac{|\cS|}{|E(G)|-1} = \frac{2^{n-|S|}}{|E(G)|-1} = \frac{2^{n- |V(G)|}}{|E(G)|-1}.$$

\smallskip
\noindent
{\bf Proof of Theorem \ref{non-uniform}/2.}
Let  $G$ be a graph with maximum degree at most $2$.
Consider $\cX$, the set of all subsets of $[n]$ of size at least ${n+|V(G)|\over 2}$. Then any two sets from $\cX$ intersect in at least $|V(G)|$ elements.
We claim that any $|E(G)|$ sets from $\cX$ form a Berge-$G$ graph.
Assume that $G$ is a cycle on $k$ vertices, for other graphs of maximum degree at most $2$, the argument is similar.
Consider an arbitrary family   $X_1, \ldots, X_k$ of sets from $\cX$.
Pick  vertices  $x_1\in X_1\cap X_2$,  $x_2\in (X_2\cap X_3) \setminus \{x_1\}$, and so on, $x_i\in (X_i\cap X_{i+1})\setminus \{x_1, \ldots, x_{i-1}\}$, $i=2, \ldots, k-1$.
Finally pick $x_k \in (X_k\cap X_1) \setminus \{x_1, \ldots, x_{k-1}\}$. Since pairwise intersections have size at least $k$, it is always possible to make such a choice of $x_1, \ldots, x_k$.
Then $\{x_1, \ldots, x_k\}$ forms an underlying  vertex set of a  Berge-$C_k$.

Thus in a coloring of $\cX$ with no monochromatic Berge-$G$, there are at most $|E(G)|-1$ sets of the same color. This implies that the number of colors is at least
$$\frac{|\cX|}{|E(G)|-1} \geq \frac{2^{n-1} - c|V(G)| \frac{2^n}{\sqrt{n}} }{|E(G)|-1} \geq   \frac{2^{n-1}}{|E(G)|-1} - o(2^n).$$
{\bf Proof of Theorem \ref{triangle}/3 - lower bound.} The lower bound for Berge-$K_3$ follows from the previous proof.

\smallskip
\noindent
{\bf Proof of Theorem \ref{triangle}/4.} When $n$ is odd, consider all sets of size at least $(n+1)/2$. We claim that any three of those, say $A_1,A_2,A_3$, form a Berge-$K_3$
hypergraph. Indeed, the three sets $A_i\cap A_j$ ($1\le i<j\le 3$) have distinct representatives by checking Hall's condition. It is obvious that none of them is empty and the union of any two of them has at least two elements. Moreover, their union has at least three elements except for $n=5$ (when three sets of size three can intersect in two elements).

Thus in any coloring of $2^{[n]}$ with no monochromatic copy of Berge-$K_3$ hypergraph, there are at most two subsets of size at least $(n+1)/2$ that have the same color. Thus the total number of colors in such a coloring is at least
$|\{A: ~ |A|\geq (n+1)/2, A \subseteq [n]\}|/2 = 2^{n-2}.$ The upper bound follows from Theorem \ref{non-uniform}/1.

\smallskip
\noindent
{\bf Proof of Theorem \ref{non-uniform}/3.} To  prove that $f(n,\B(2K_2)) = 2^n - \binom{n}{2}-n-1$,  note that in any coloring of subsets of size at least three in $[n]$ without monochromatic $\B(2K_2)$, all subsets must have distinct colors. Thus  $$f(n,\B(2K_2)) \geq |\{X: X\subseteq [n], |X|\geq 3\}| =
2^n - \binom{n}{2} -n -1.$$  On the other hand, the following coloring has no monochromatic $\B(2K_2)$:
color all sets of size at least $3$ with distinct colors, color each set of at most two elements with the color of some $3$-element set containing it. Then each color class of size at least two consists of subsets of some three element set, so it does not contain a copy of $\B(2K_2)$ and the number of colors is $2^n - \binom{n}{2} - n -1$.

\section{Ramsey number of the Berge triangle}\label{sec:triangle}
In this subsection we set $\cF=\B(K_3)$.

\smallskip
\noindent
{\bf Proof of Theorem \ref{triangle}/1,2 - upper bounds.}
%
%
Consider a coloring of $K_n^r$, an $r$-uniform $n$-clique with $k$ colors without monochromatic member of $\cF$, $n= R_r(\cF, k)-1$.
Then each color class has at most ${\rm ex}_r(n, \cF)$ edges. A result of Gy\H{o}ri \cite{G} implies that  ${\rm ex}_r(n, \cF)\leq \frac{n^2}{8(r-2)}.$
Thus $\binom{n}{r}$,   the total number of hyperedges,  is at most $ k n^2/(8(r-2))$.
This provides the general upper bounds  and all upper bounds for the Ramsey numbers for $i$ colors, $i=2,3,4,5,6,7,8$. \\

\smallskip
\noindent
{\bf Proof of Theorem \ref{triangle}/1 - lower bound.}
We shall instead provide an upper bound on the number of colors needed to color the triples on $n$ vertices so that no monochromatic Berge-$K_3$ is created.
We split the vertex set in two almost equal parts, $A$ and $B$.
Let $A_1, A_2,  \ldots$ and $B_1, B_2 \ldots$
be color classes of optimal  proper  edge-colorings of  complete graphs (uniformity $2$) on a  vertex set $A$ and on a vertex set $B$, respectively.
Let $\cA_i = \{ \{x,y,z\}:  \{x,y\} \in A_i, z\in B\}$, let $\B_i =  \{ \{x,y,z\}:  \{x,y\} \in B_i, z\in A\}$.
Then we see that each of $\cA_i$'s and $\B_i$'s does not contain a member or $\cF$. Moreover these classes of triples contain all hyperedges of $K_n^3$ with vertices in both $A$ and $B$. Color all triples in $\cA_i$ with color $a_i$, all tripes in $\B_i$ with color $b_i$, for distinct $a_i$'s and $b_i$'s.
Further, color the triples with all elements in $A$ or with all elements in $B$ recursively using new colors such that the set of colors used on triples from $A$ is the same as the set of colors used on triples from $B$.
Let $f(n)$ be the number of colors given by this construction and $\chi'(G)$ denote the chromatic index of a graph $G$.
Then
\begin{eqnarray*}
f(n) & \leq & \chi'(K_{\lfloor n/2 \rfloor}) + \chi'(K_{\lceil n/2 \rceil}) + f(\lceil n/2 \rceil)\\
&\leq & \lfloor n/2 \rfloor+ \lceil n/2 \rceil+ f(\lceil n/2 \rceil)\\
&= & n+ f(\lceil n/2 \rceil)\\
&\leq & 2n + \log n.
\end{eqnarray*}

So, we have that the number of colors $k$ is bounded as $k\leq 2n + \log n$,
thus $n\geq k/2 - o(k)$.

\smallskip
\noindent
{\bf Proof of  Theorem \ref{triangle}/2 - lower bounds.}
The lower bound for $2$ and $3$ colors is obvious since  two edges of $K_4^3$ can be colored red and  the other two  blue. An
$\cF$-free $4$-coloring of $K_5^3$ on $[5]$ can be  given by splitting the edge set into color classes as follows:
$$123,124,125; ~134,234,345;~135,145; ~235,245.$$
Note that for each color class there is a pair of vertices that belongs to each hyperedge of this class, thus there is no monochromatic  Berge-$K_3$ hypergraph.
An $\cF$-free $5$-coloring of $K_6^3$ on $[5]\cup \infty$ can be  given so that the first color class is $$\infty12,\infty13,345,245$$ and each other color class is obtained from the first one by cyclically shifting the vertex labels that are not $\infty$  and keeping $\infty$ fixed.  Finally, an $\cF$-free $6$-coloring of $K_7^3$ on $[5]\cup \infty_1 \cup \infty_2$ can given so that the first color class consists of $5$ edges:
$\infty_1\infty_21$,  $\infty_1\infty_2 2$, $\infty_1\infty_23$, $\infty_1\infty_2 4$, $\infty_1\infty_2 5$.
The second color class is  $$125,134,\infty_125,\infty_225,\infty_134,\infty_234,$$  and third  through sixth color classes are obtained from the second by  keeping $\infty_1,\infty_2$ fixed and cyclically shifting other vertex labels.
The lower bound for $8$ colors comes from the general construction.

\section{Ramsey results for Berge-$K_4$}\label{sec:quadrangle}

In this section, set $\cF=\B(K_4)$.

\smallskip
\noindent
{\bf Proof of Theorem \ref{quadrangle}/1 - lower bound.}
A natural lower bound on $R_3(\cF,k)$ comes from covering the edges of $K_n^3$ with the smallest possible number of $3$-partite subhypergraphs. Indeed, a $3$-partite $3$-uniform hypergraph cannot contain any member of $\cF$, thus $f(n)$, the minimum number of $3$-partite hypergraphs needed to cover all edges of $K_n^3$ provides a coloring with $f(n)$ colors containing no monochromatic member from $\cF$.  This is a well studied problem in coding theory, a special perfect hash family.  Apart from very small $n$, exact values of $f(n)$ are not known, only upper bound tables are available \cite{CW}. The order of magnitude of $f(n)$ is known, $c_1\log(n)\le f(n)\le c_2\log(n)$ \cite{FK}. The upper bound implies
that  $2^{ck} \le R_3(\cF,k)$ for a positive constant $c$.\\

Interestingly, the upper bound of $f(n)$ is easy from probabilistic constructions, however, simple explicit constructions are not known (for general $n$). It seems worthwhile to give a very simple construction leading to a $2\log_2^2(n)$ upper bound.
%
Split the vertex set of $K_n^3$ into two almost equal parts $A$ and $B$,
  consider colorings $c_A,c_B$ on the pairs of $A,B$ with no monochromatic $K_{3}$ and with disjoint color sets. Extend this coloring to edges of $K_n^3$ intersecting both $A$ and $B$ as follows: edges $xyz$ with $x,y\in A,z\in B$ are colored with $c_A(xy)$, edges $xyz$ with $x\in A,y,z\in B$ are colored with $c_B(yz)$.  This can be easily done by using  no more than $2\log_2(n)$ colors. The uncolored edges, i.e. edges inside $A$ and inside $B$ can be colored recursively, using the same set of new colors. This leads to the recursive bound $f(n)\le 2\log_2(n)+f(\lceil n/2 \rceil)$.\\

\smallskip
\noindent
{\bf Proof of Theorem \ref{quadrangle}/1 - upper bound.}
The inequality ~ $R_3(\cF, k)\leq   R(K_4^*(v), k)$ follows from Theorem \ref{shadowbound}. Thus Theorem \ref{quadrangle} follows from the following lemma.

\begin{lemma}\label{K*} For any  $\epsilon$,  $0<\epsilon<1/4$, and any $k\geq 1$,
$R(K_4^*(v), k)\leq (1+\epsilon)^k \epsilon^{-1}  k!$.
In particular,   $R(K_4^*(v), k)\leq (k-1)e(1+o(1))k!. $
\end{lemma}

\begin{proof}
We shall prove the statement by induction on $k$ with a trivial basis for $k=1$.
Consider a coloring of $E(K_n)$,  with $k$ colors and no monochromatic copy of a member from $K_4^*(v)$. Note that from each monochromatic triangle $T$ of a fixed color class $i$ we can select {\em one}  vertex $v_i(T)$ of degree two in color class $i$. Let $X_i=\cup v_i(T)$, where the union is taken  over all monochromatic triangles of color $i$.
Then $X_i$  is an independent set in color $i$, i.e. cannot contain any edge of color $i$.  By induction
$|X_i| \leq (1+\epsilon)^{k-1} \epsilon^{-1} (k-1)!.$\\

{\bf Case 1.} $|X_i|> n/((1+\epsilon)k)$ for some $i$.   \\
Then $n\leq (1+\epsilon) k |X_i|\leq  (1+\epsilon)k  (1+\epsilon)^{k-1} \epsilon^{-1} (k-1)! = (1+\epsilon)^k \epsilon^{-1} k!$ by induction.\\

{\bf Case 2.}  $|X_i|\leq n/ ((1+\epsilon) k)$ for each $i$.
Note that deleting all $X_i$'s leaves vertex set $X'$ of size at least $n- n/(1+\epsilon) = n \epsilon /( 1+ \epsilon) $ such that $X'$ does not contain any monochromatic triangles. Then $n\leq (1+\epsilon) |X'| /\epsilon \leq (1+ \epsilon) R(K_3, k) /\epsilon \leq k! (1+\epsilon)/\epsilon \leq (1+\epsilon)^k \epsilon^{-1} k!$.\\

Optimizing over $\epsilon$, for large $k$ we see that for $\epsilon = 1/(k-1)$, $n \leq (k-1)e(1+o(1))k!.$
\end{proof}

\smallskip
\noindent
{\bf Proof of Theorem \ref{quadrangle}/2.}  Note that Theorem \ref{shadowbound} gives
\begin{equation}\label{notenough}
R_3(\cF,2)\le R(K_4^*(v),2)+1
\end{equation}
and the Ramsey number on the right can be determined easily.

\begin{lemma}\label{rk3plus}  $R(K_4^*(v),2)=7$.
\end{lemma}
\begin{proof}
Set $\cF=\B(K_4^*(v))$. For the lower bound $R(\cF,2)\ge 7$, consider $K_{3,3}$ and its complement as a $2$-coloring on $K_6$.

For the upper bound, suppose  $K=K_7$ is colored with red and blue.  A well-known result of Goodman \cite{GO} implies that any $2$-colored $K_7$ contains at least four monochromatic triangles, among them two of the same color, say $T_1,T_2$ are vertex sets of  red triangles.

Suppose for contradiction that we have no monochromatic member of $\cF$.  This implies that there exist $v_1\in T_1,v_2\in T_2$ with red degree two in $K$. Consequently the edge $v_1v_2$ and all edges incident to $v_1,v_2$ and not on $T_1,T_2$ are blue. Set
$$T=(T_1\cup T_2)\setminus \{v_1,v_2\}, ~S=V(K)\setminus (T_1\cup T_2).$$ Easy inspection shows that either there is a blue member of $\cF$ with base triangle $v_1,v_2,s$ with $s\in S$ or all edges of $[S,T]$ are red. In the latter case any red triangle with a red edge in $T^*$ and with any $s\in S$ is a base triangle of a red member of $\cF$, apart from one case: when $|S|=3, |T|=2$. In this exceptional case a red edge $s_1s_2$  gives a red member of $\cF$ (with base triangle $s_1\cup T$) and a blue edge $s_1s_2$ gives a blue member of $\cF$ (with base triangle $s_1s_2v_1$).
\end{proof}

Unfortunately,  Lemma \ref{rk3plus} implies through (\ref{notenough}) only $R_3(\cF,2)\le 8$, to decrease it by two seems to require more difficult argument.
First we need the Tur\'an number of $\cF$ for $n=5$.


\begin{lemma}\label{exbk4} ${\rm ex}_3(5, \cF)=5$.
\end{lemma}
\begin{proof} Five edges clearly do not form $\cF$ thus we have to show ${\rm ex}_3(5,\cF)< 6$. Assume $H$ is a $3$-uniform hypergraph with six edges on a vertex set $[5]$ without any member of $\cF$. Observe that the maximum vertex degree of $H$ is at least $\lceil {6\times 3\over 5}\rceil=4$.
\begin{itemize}
\item Suppose that for some $1\le i<j\le 5$, $ij$ is not covered by any edge of $H$. By symmetry, let $i=1,j=2$. Then $H$ either contains the six edges meeting $\{1,2\}$ in one vertex or one of them, say $234$ is missing. In the first case the assignment $$e_{13}=134,e_{14}=145,e_{15}=135,e_{34}=234,e_{35}=235,e_{45}=245$$ defines a $\cF$, otherwise the assignment $e_{34}=234$ is replaced by the assignment $e_{34}=345$ to get a $\cF$. In both cases we have a contradiction. Thus all pairs of vertices are covered by some edge of $H$. Assume that vertex $1$ has maximum degree in $H$.

\item   If $d(1)=6$ then the edges containing $1$ can be obviously assigned to pairs of $\{2,3,4,5\}$ to form a member of $\cF$. Similarly, if $d(1)=5$ then the five edges containing $1$ with the edge covering the  yet uncovered pair of the link of $1$ form a member of $\cF$ on $\{2,3,4,5\}$. Both cases yield contradiction.

\item  If $d(1)=4$ then the link of $1$ is either the four cycle $2,3,4,5,2$ or the graph with edges $23,34,24,45$. In the first case we have a member of $\cF$ on $\{2,3,4,5\}$ with the assignment $$e_{23}=123,e_{34}=134,e_{45}=145,e_{25}=125$$
    extended by the edges covering the uncovered pairs $24,35$. In the second case we make the assignments $$e_{23}=123,e_{24}=124,e_{34}=134,e_{45}=145.$$ If there exist two distinct edges of $H$ to cover the yet uncovered pairs $25,35$ we get a member of $\cF$ on vertex set $\{2,3,4,5\}$.  Otherwise $234,235$ are both edges of $H$ and we have a member of $\cF$ on $\{1,2,3,4\}$ by the assignments $$e_{12}=123,e_{13}=134,e_{14}=145,e_{23}=235,e_{24}=124,e_{34}=234,$$
    a contradiction, finishing the proof. \end{itemize}
\end{proof}


We are ready to prove that $R_3(\cF, 2)=6$.   It is easy to see that $R_3(\cF,2)\ge 6$ since in any $2$-coloring of $K_5^3$ with five edges in each  color class there is no monochromatic member of $\cF$.

For the upper bound consider a $2$-coloring $c$ of the edges of $K=K_6^3$ with no monochromatic member of $\cF$.  Let $H$ be the hypergraph defined by the edges of the majority color, containing at least $10$ hyperedges. By Lemma \ref{exbk4}, any $5$ vertices of $H$ induce at most $5$ hyperedges, thus the remaining at least $|E(H)|-5$
hyperedges contain the sixth vertex, i.e., $d(v)\ge |E(H)|-5$ for every $v\in V(H)$ implying
\begin{equation} \label{upper-bound}
6(|E(H)|-5)\le \sum_{v\in v(H)} d(v)=3|E(H)|,
\end{equation}
that in turn implies that  $|E(H)|\le 10$. On the other hand $|E(H)|\ge 10$, thus $|E(H)|=10$ and so the inequality in (\ref{upper-bound}) hold as equality. In particular, $H$ is $5$-regular, implying the same for the other color. Thus we may assume that $K$ is partitioned into a red and a blue hypergraph $H_r,H_b$ both $5$-regular.

Let $v$ be an arbitrary vertex of $H$ and consider the hypergraphs induced by $H_r,H_b$ on $V-v=[5]$. By Lemma \ref{exbk4} both contain $\B(K_4-e)$, we show that some of them can be extended to $\cF$ by adding a suitable edge of $K$ containing $v$. It is convenient to represent $H_r, H_b$ with the graphs $G_r, G_b$ of their complements (with their inherited colorings).  Apart from color changes we have four possible cases.

\noindent
{\bf Case 1. } $G_r,G_b$ are complementary five cycles. Assume the edges of $G_r$ are $i,i+1$, thus  edges $i,i+1,i+2$ belong to $H_r$.  Every four element subset form a $\B(K_4-e)$ in $H_r$, for example the edges $$e_{12}=123,e_{45}=145,e_{25}=125,e_{24}=234,e_{35}=345$$ cover all pairs of $\{2,3,4,5\}$ except $\{3,4\}$. Thus we have a red $\cF$ unless $\{v,i,i+1\}$ are all forming blue edges in $K$. The same argument gives that $\{v,i,i+2\}$ must be red edges of $K$. But then there are many monochromatic $\cF$s, for example the assignment $$e_{12}=123,e_{45}=145,e_{25}=125,e_{24}=v24,e_{35}=v35, e_{45}=345$$ gives one on $\{2,3,4,5\}$.

\noindent
{\bf Case 2. } $G_r,G_b$ are complementary bulls.  Assume that the edges of $G_r$ are $12,23,34,24,45$, implying that the edges $345,123,145,125,135$ belong to of $H_r$ and the edges  $234,235,134,245,124$ belong to $H_b$.  Then we have two $\B(K_4-e)$s in $H_r$ on $\{1,2,3,5\}$. One of them with $$e_{12}=123,e_{13}=135,e_{15}=145,e_{25}=125,e_{35}=135$$ implying that $v23$ must be blue. The other is  $$e_{13}=135,e_{15}=145,e_{23}=123,e_{25}=125,e_{35}=135$$
implying that $v12$ must be blue. However, then vertex $2$ has degree at least six in $H_b$, contradiction.

\noindent
{\bf Case 3. } $G_r=K_4-e$, $G_b$ is its complement.  Assume that the edges of $G_r$ are $12,23,34,14,24$, then edges $125,135,145,235,345$ belong to $H_r$ and  edges $123,124,134,234,245$ belong to $H_b$. We have three $\B(K_4-e)$s in $H_b$ on vertices $\{1,2,3,4\}$. All use assignments $$e_{23}=234,e_{24}=245,e_{34}=134$$ and they can be extended to $\B(K_4-e)$ with $e_{13}=123,e_{14}=124$,  or $e_{12}=123,e_{14}=124$ or $e_{12}=124,e_{13}=123$,  respectively. Since we have no $\B(K_4)$ in $H_b$, the edges $12v,13v,14v$ are all red and vertex $1$ has degree at least six in $H_r$, contradiction.

\noindent
{\bf Case 4. } $G_r$ is a four-cycle with a pendant edge, $G_b$ is its complement.  Assume that the edges of $G_r$ are $12,23,34,45,25$, then edges $123,125,145,134,345$ belong to $H_r$ and  edges $124,135,234,235,245$ belong to $H_b$. Then we have two $\B(K_4-e)$s in $H_b$ on $\{1,3,4,5\}$. We can assign in both $$e_{14}=124,e_{34}=234,e_{35}=235,e_{45}=245$$ and complete it with either $e_{13}=135$ or $e_{15}=135$. Since we have no $\B(K_4)$ in $H_b$, $13v,15v$ are both red edges, consequently vertex $1$ has degree at least six in $H_r$, contradiction.

\eject

\end{document}